\documentclass{article}

\usepackage{amsthm}
\usepackage{amsmath,latexsym,amsfonts,amssymb,graphicx}
\usepackage{graphicx}[dvipdfmx]
\usepackage{color}
\usepackage[margin=20mm]{geometry}

\newcommand{\Ra}{\mathcal{R}}
\newcommand{\Nu}{\mathcal{N}}

\newcommand{\R}{\mathbb{R}}

\def\vec#1{\mbox{\boldmath $#1$}}

\newtheorem{theorem}{Theorem}
\newtheorem{lemma}{Lemma}
\newtheorem{remark}{Remark}
\newtheorem{corollary}{Corollary}

\newtheorem{definition}{Definition}

\newtheorem{assumption}{Assumption}

\begin{document}

\title{Strong consistency of an estimator by the truncated singular value decomposition
for an errors-in-variables regression model with collinearity}

\author{Kensuke Aishima
\thanks{Faculty of Computer and Information Sciences,  
  Hosei University, Tokyo 184-8585, Japan
  (\texttt{aishima@hosei.ac.jp}).}
}

\maketitle

\begin{abstract}
In this paper, we prove strong consistency of an estimator
by the truncated singular value decomposition
for a multivariate errors-in-variables linear regression model with collinearity.
This result is an extension of Gleser's proof of the strong consistency of total least squares solutions to the case with modern rank constraints. While the usual discussion of consistency in the absence of solution uniqueness deals with the minimal norm solution, the contribution of this study is to develop a theory that shows the strong consistency of a set of solutions. The proof is based on properties of orthogonal projections, specifically properties of the Rayleigh-Ritz procedure for computing eigenvalues. This makes it suitable for targeting problems where some row vectors of the matrices do not contain noise. Therefore, this paper gives a proof for the regression model with the above condition on the row vectors, resulting in a natural generalization of the strong consistency for the standard TLS estimator.

\end{abstract}

\noindent
{\em Keywords}: Singular value decomposition, eigenvalue problems, total least squares, low rank approximation, consistent estimation

\noindent
{\em PACS}: 65F20, 65F15, 15A09, 15A18, 62H12, 60B12 

\section{Introduction}
This study deals with asymptotic theory for multivariate errors-in-variables linear regression models. The models are associated to an overdetermined system $AX \approx B$ for given $A\in \R^{m\times n}$ and $B\in \R^{m \times \ell}$, contaminated with random noise, where $X\in \R^{n \times \ell}$ is to be estimated. This mathematical formulation has been successfully applied to many scientific and engineering research areas, in particular to a data fitting problem. In such a problem setting, with more information available in recent years, rank constraints are often introduced to overcome over-fitting that causes a loss of accuracy in parameter estimation.
We aim to make progress in the asymptotic theory of the effects of the rank constraints. More specifically, we derive an important relationship between a multivariate errors-in-variables linear regression model associated with overdetermined system $AX \approx B$ and the solution of the total least squares (TLS) problem with rank constraints.
With rank constraints, the solutions are usually not unique and form a set. Therefore, the minimum norm solution is often discussed.
In contrast, it is worth noting that this paper proves strong consistency in the statistical sense 
for a set of solutions to the TLS problem under reasonable assumptions.

From the point of view of numerical methods, 
this paper focuses on the total least squares (TLS) method~\cite{GV1980,GV2013}, which is described as follows.
It is a powerful technique to find $X$ for the overdetermined system $AX \approx B$.
The solution $X$ is computed by orthogonal projections and the singular value decomposition (SVD)
due to the Eckart-Young-Mirsky theorem.
Among many studies, \cite{Hnetynkova2011} gives the first full solvability analysis of multiple observation TLS covering all cases.
The TLS has many variants
adopted to the problem formulations in many application areas.
A typical example is a rank-constrained TLS problem where the coefficient matrix $A$ has collinearity.
This problem can be solved by
the so-called truncated SVD~\cite{FB1994,Huffel1987,MDB2021}, and statistical consistency is presented in~\cite{Park2011}.
A similar constraint, the reduced rank estimation, which assumes low rank of $X$, is also important and often used~\cite{MCWZ2015}.
In addition to the rank constraints, we are concerned with the following
constrained TLS problems. Suppose that,
for the matrix $A$, the first $j$ columns are exactly known
without noise. In addition, the first $k$ rows of $A$ and $B$
are exactly known without noise.
Such a constrained TLS problem can be solved
by the SVD and the QR decomposition~\cite{Demmel1987,Huffel1991,wei1998}.
The discussion of the column constraints
is detailed in \cite{GHS1987,Huffel1989,PW1993}.
Recently, the problem with the row constraints 
has attracted much attention in application areas,
and thus numerical method is revived with
mathematical analysis for the solvability conditions~\cite{LJ2022,LJY2022}.
The importance of condition numbers has led to vigorous research in recent years~\cite{MZW2020,MDB2021,ZMW2017}.

In this paper, we perform statistical consistency analysis for large sample size.
More specifically, we generalize the strong consistency of the TLS estimator,
proved by Gleser~\cite{Gleser1981}.
In the previous study,
the strong consistency is extended to the minimal norm solution to the rank deficient case~\cite{Park2011}.
Similarly, Gleser's result is easily extended to
the column constraints TLS problem as in~\cite[\S 3]{Huffel1991}.
Recently, for the row constraints TLS problem,
the strong consistency has been proved
under reasonable assumptions~\cite{Aishima2022}.
This consistency proof is generalized to
the case of both column and row constraints in \cite{Aishima2023}.
There exists other consistency analysis for some variants of the TLS~\cite{KH2004,KMH2005}.
However, the proof of strong consistency in the previous study
is restricted to the unique solution
or the minimal norm solution in the rank deficient case,
as far as the author knows.

The contribution of this paper is to prove strong consistency 
for a set of solutions to the TLS problem under reasonable assumptions.
More specifically, we prove the convergence of a subspace
containing any solution vector with probability one.
If the coefficient matrix $A$ is exact with no errors, it is easy to identify the solution set from the null space. However, in our problem setting, the coefficient matrix $A$ is also subject to perturbation, and thus the null space is unclear. Therefore, although it is possible to identify the null space by a method other than the TLS, this paper points out that the null space can be naturally identified from a subspace computed by the TLS.
The idea on using subspaces is basic and simple.
Similar discussions concerning
the range null space decomposition
are presented in~\cite[\S 3]{KMH2005}. However, 
this paper is the first to organize the theory into a rigorous proof of strong consistency.
The proof is based on properties of the Rayleigh-Ritz procedure for computing eigenvalues. 
This makes it suitable for targeting problems where the first $k$ row vectors of $A$ and $B$ do not contain noise. 
Therefore, we focus on the above TLS problem with rank constraints,
where the first $k$ rows of $A$ and $B$
are known exactly without noise.
Note that the special case of $k=0$
can be reduced to the standard TLS problem.

This paper is organized as follows.
The next section is
devoted to a description of the previous study to formulate our target problems.
In Section~\ref{sec:aim}, we present a regression model
and the TLS problem with rank constraints.
In Section~\ref{sec:mainresult},
we prove strong consistency of the estimator
by the TLS with rank constraints.
Finally, Section~\ref{sec:conclusion} gives the conclusion.

\section{Previous study}\label{sec:prev}
This section is devoted to a description of the previous study to formulate our target problems.
Throughout the paper, $I_{p}$ is a $p \times p$ identity matrix,
and $0_{p\times q}$ is a $p\times q$ zero matrix.
In addition, $\Nu(\cdot)$ and $\Ra(\cdot)$ mean the null space and the range of the argument matrix, respectively.

\subsection{Total least squares and an errors in variables regression model}
Given $A\in \R^{m\times n}$ and ${B} \in \R^{m\times \ell}$,
the total least squares (TLS) is to find ${X}_{\rm tls} \in \R^{n\times \ell}$
for the following optimization problem:
\begin{align}\label{eq:tls}
\min_{{X}_{\rm tls}, \Delta A, \Delta {B}} {\|\Delta A\|_{\rm F}}^{2}+{\|\Delta {B}\|_{\rm F}}^{2}
\quad {\rm such \ that}\  (A+\Delta A){X}_{\rm tls}={B}+\Delta {B}.
\end{align}

The solution vector ${X}_{\rm tls}$ can be obtained
by the singular value decomposition (SVD)
under reasonable assumptions~\cite{GV1980}
in view of the Eckart-Young-Mirsky theorem
that states the optimal low rank approximation
of the truncated SVD.
Here we describe how to compute ${X}_{\rm tls}$ 
with the use of the eigenvalue decomposition
to consider a statistical model.

First, we define $C$, $\Delta C$, and $Y$ such that
\begin{align}\label{eq:CY}
C=
\begin{bmatrix}
A & B 
\end{bmatrix}\in \R^{m\times (n+\ell)},
\quad
\Delta{C}=
\begin{bmatrix}
\Delta{A} & \Delta{B} 
\end{bmatrix}\in \R^{m\times (n+\ell)},
\quad
Y=
\begin{bmatrix}
-X_{\rm tls} \\ I_{\ell}
\end{bmatrix}\in \R^{(n+\ell)\times \ell}.
\end{align}
Then we have
\begin{align}\nonumber
(C+\Delta{C})Y=0_{m\times \ell}
\end{align}
from the constraint in the optimization problem of \eqref{eq:tls}.
For solving this optimization problem,
compute
\begin{align}\nonumber
F:=C^{\top}C\in \R^{(n+\ell)\times (n+\ell)}.
\end{align}
In addition, for $i=1,\ldots ,n+\ell$,
let $\lambda_{i}$ and $\vec{z}_{i}$ denote the eigenpairs, i.e., 
\begin{align}\nonumber
F\vec{z}_{i}=\lambda_{i}\vec{z}_{i}\ (i=1,\ldots ,n+\ell),\quad
(0\le )\lambda_{1}\le \cdots \le \lambda_{n+\ell}.
\end{align}
Moreover, for a part of the eigenpairs, 
let
\begin{align}\nonumber
\Lambda_{\ell}={\rm diag}(\lambda_{1}, \ldots , \lambda_{\ell}),\quad
Z_{\ell}=[\vec{z}_{1},\ldots ,\vec{z}_{\ell}].
\end{align}
Furthermore, let the eigenvector matrix $Z_{\ell}$ be divided into
\begin{align}\nonumber
Z_{\ell}=:
\begin{bmatrix}
Z_{\ell,{\rm upper}}
\\
Z_{\ell,{\rm lower}}
\end{bmatrix},\ 
Z_{\ell,{\rm upper}}\in \R^{n\times \ell},\
Z_{\ell,{\rm lower}}\in \R^{\ell \times \ell}.
\end{align}
Then we have
\begin{align}\nonumber
X_{\rm tls}=-Z_{\ell,{\rm upper}}{Z_{\ell,{\rm lower}}}^{-1},
\end{align}
where the theoretical background is described in~\cite{Gleser1981,GV1980} and~\cite[\S 6.3]{GV2013} based on the truncated SVD.

Although the TLS problems are not always solvable, conditions for solvability and 
their classification are discussed in~\cite{Hnetynkova2011}.
The classification is due to the SVD of $C$ as follows.
Let $C=:U\Sigma V^{\top}$ denote the SVD,
where $U,V$ are orthogonal matrices and, $\Sigma$
is a diagonal matrix: $\Sigma={\rm diag}(\sigma_{1},\ldots ,\sigma_{n+\ell})$
for $\sigma_{1}\geq \cdots \geq \sigma_{n+\ell}$.
Note that $V$ is an eigenvector matrix of $F=C^{\top}C$.
If $\sigma_{n}>\sigma_{n+1}$, then the TLS problem has a unique solution $X_{\rm tls}$.
In the case of $\sigma_{n}=\sigma_{n+1}$, the features of the solutions
are classified as in \cite{Hnetynkova2011}.
Under the assumption of our statistical problem setting,
the spectral gap associated with $\sigma_{n}>\sigma_{n+1}$
exists for sufficiently large $m$,
where such solvability conditions are satisfied.

Our statistical problem setting
relevant to \eqref{eq:tls}
is described as follows.
Let $\bar{A}\in \R^{m\times n}$ and $\bar{{B}}\in \R^{m\times \ell}$
denote parameters depending on $m$.
More specifically,
\begin{align*}
\bar{A}:=[\bar{\vec{a}}_{1},\ldots ,\bar{\vec{a}}_{m}]^{\top},\quad
\bar{{B}}:=[\bar{\vec{b}}_{1},\ldots ,\bar{\vec{b}}_{m}]^{\top},
\end{align*}
where $\bar{\vec{a}}_{i}\in \R^{n}$ and $\bar{\vec{b}}_{i}\in \R^{\ell}$
for $i=1,\ldots ,m$.
In the statistical sense,
$\bar{\vec{a}}_{i}\ (i=1,\ldots ,m)$
correspond to the explanatory variables,
$\bar{\vec{b}}_{i}\ (i=1,\ldots ,m)$
to the objective variables,
and $m$ to the sample size.
Then we assume
${\bar{\vec{a}}_{i}}{}^{\top}{X}=\bar{\vec{b}}_{i}{}^{\top}\ (i=1,\ldots m)$
for some ${X}\in \R^{n\times \ell}$.
The matrix $X$ corresponds to the weight parameters in the linear regression model. 
The errors-in-variables linear regression model
is formulated as
\begin{align}\label{eq:tlsregression}
\bar{A}{X}=\bar{{B}},\quad 
A=\bar{A}+E_{A},\quad 
{B}=\bar{{B}}+{E}_{{B}},
\end{align}
where all the elements of $E_{A}$ and ${E}_{{B}}$
are random variables, and only $A$ and ${B}$ can be observed.
Hence, letting
\begin{align}\label{eq:E}
E:=
\begin{bmatrix}
 E_{A} &  {E}_{{B}} 
\end{bmatrix},
\end{align}
we assume the following.
\begin{assumption}\label{as:E}
The rows of $E$ are i.i.d. with common mean row vector $\vec{0} \in \R^{n+\ell}$ and common covariance matrix $\sigma^2 I_{n+\ell}$. 
\end{assumption}

The above assumption is not restrictive in the following sense.
If the common covariance matrix is a general form $\Sigma \in \R^{(n+\ell)\times (n+\ell)}$,
we can reduce such a problem to the simple problem under Assumption~\ref{as:E}
with the use of the Cholesky factorization $LL^{\top}:=\Sigma$~\cite[\S 5]{Gleser1981}.
From the viewpoints of numerical computations,
since the generalized eigenvalue problem covers the general covariance matrix $\Sigma$,
it is possible to design an efficient algorithm avoiding numerical errors
as in~\cite{Huffel1989}.
Our purpose is asymptotic analysis in statistical sense. 
To construct efficient numerical methods is beyond the scope of this paper.
In the following, we perform consistency analysis under Assumption~\ref{as:E},
where the standard deviation $\sigma$ is unknown.

In addition, assume that
$\lim_{m\to \infty}m^{-1}\bar{A}^{\top}\bar{A}$ exists, and it is positive definite,
related to the uniqueness of ${X}$ in the asymptotic regime as $m \to \infty$.
Below we rewrite the assumption
for extending the regression model. 
Let
\begin{align}\label{eq:barC}
\bar{C}:=
\begin{bmatrix}
\bar{A} & \bar{{B}} 
\end{bmatrix}=
\begin{bmatrix}
\bar{A} & \bar{A}{X} 
\end{bmatrix}.
\end{align}
From easy calculations, we have
\begin{align*}
m^{-1}\bar{C}^{\top}\bar{C}=
\begin{bmatrix}
\ m^{-1}\bar{A}^{\top}\bar{A} & m^{-1}\bar{A}^{\top}\bar{A}{X} \\
\ m^{-1}{X}^{\top}\bar{A}^{\top}\bar{A} & m^{-1}{X}^{\top}\bar{A}^{\top}\bar{A}{X}
\end{bmatrix}.
\end{align*}
The assumption that $\lim_{m\to \infty}m^{-1}\bar{A}^{\top}\bar{A}$ is
positive definite is mathematically the same as the next assumption.
\begin{assumption}\label{as:CTC}
$\lim_{m\to \infty}m^{-1}\bar{C}^{\top}\bar{C}=:\bar{S}$ exists, and ${\rm rank}(\bar{S})=n$. 
\end{assumption}

Then we have the next theorem
that states the strong consistency of the estimation ${X}_{\rm tls}$ in \eqref{eq:tls}.
\begin{theorem}[Strong consistency~{\cite[Lemma~3.3]{Gleser1981}}]\label{thm:Gleser}
Under Assumptions~\ref{as:E} and~\ref{as:CTC},
we have
\begin{align*}
\lim_{m\to \infty}{X}_{\rm tls}={X} \ \text{with probability one},
\end{align*}
where ${X}_{\rm tls}$ is the solution of \eqref{eq:tls}.
\end{theorem}

In general, if an estimator based on random variables
converges to a target value with probability one, the estimator has strong consistency.
The rigorous definition is described in the next section.
This section is devoted to a description of the previous study 
for extensions of the TLS problem and the corresponding consistency analysis.

\subsection{Rank deficient case}
Here we consider the case where
\begin{align}\label{eq:rankdeficientA}
{\rm rank}(\bar{A})=r\le n
\end{align}
for $\bar{A}$ in~\eqref{eq:tlsregression}.
This is the general case where the coefficient matrix $\bar{A}$ is not necessarily column full rank.
For the consistency analysis, we assume the following.
\begin{assumption}\label{as:CTC2}
$\lim_{m\to \infty}m^{-1}\bar{C}^{\top}\bar{C}=:\bar{S}$ exists. Let ${\rm rank}(\bar{S})=r\le n$. 
\end{assumption}
The corresponding TLS problem with the rank constraint
is as follows.
\begin{align}\label{eq:tls2}
\min_{{X}_{\rm tls}, \Delta A, \Delta {B}} {\|\Delta A\|_{\rm F}}^{2}+{\|\Delta {B}\|_{\rm F}}^{2}
\quad {\rm such \ that}\  (A+\Delta A){X}_{\rm tls}={B}+\Delta {B},\ {\rm rank}(A+\Delta A)=r.
\end{align}

\begin{remark}
The above rank constrained optimization problem~\eqref{eq:tls2} is the so called truncated TLS (TTLS) problem~\cite{FB1994,Huffel1987,MDB2021,Park2011}, a contrivance for $A$ with collinearity. In real TTLS, the rank is determined based on the singular value distribution of the matrix $A$ or $C=[A \ \ B ]$ and the nature of the singular vectors, which is not explicitly written in the optimization problem. However, in the statistical problem setting of this paper, the rank condition is explicitly stated in the optimization problem as above, because it is equivalent to estimating the rank with a certain value $r$.
\end{remark}

Due to the rank deficient matrix $A+\Delta A$,
the solution vector $X_{\rm tls}$ is not unique.
Thus, let $X_{\rm tls}^{*}$ denote the minimal norm solution.
Similarly, $X$ in~\eqref{eq:tlsregression} is not unique.
Let $X_{\min}$ denote the minimal norm solution.
Then we have the next theorem
that states the strong consistency
of the minimal norm solution.

\begin{theorem}[Strong consistency~{\cite[Lemma~4.6]{Park2011}}]\label{thm:park}
Under Assumptions~\ref{as:E} and~\ref{as:CTC2},
we have
\begin{align*}
\lim_{m\to \infty}{X}_{\rm tls}^{*}={X}_{\min} \ \text{with probability one},
\end{align*}
where ${X}_{\rm tls}^{*}$ and $X_{\min}$ are the minimal norm solutions of \eqref{eq:tls2}
and \eqref{eq:tlsregression}, respectively.
\end{theorem}

The above theorem is a natural extension of
the full rank case to the rank deficient case.

\subsection{Constrained total least squares problems and their consistency analysis}
\label{ssec:ctls}
Here, we discuss another constraint in the TLS problem.
Let $A$ be divided into
\begin{align}\nonumber
A=\begin{bmatrix}
{A}_{1} \\
{A}_{2}
\end{bmatrix},\quad
{A}_{1}\in \R^{k\times n},\quad {A}_{2}\in \R^{(m-k)\times n}.
\end{align}
Similarly,
\begin{align}\nonumber
B=\begin{bmatrix}
{B}_{1} \\
{B}_{2}
\end{bmatrix},\quad
{B}_{1}\in \R^{k\times \ell},\quad {B}_{2}\in \R^{(m-k)\times \ell}.
\end{align}

We consider the following optimization problem:
\begin{align}\label{eq:ctls2}
\min_{{X}_{\rm ctls}, \Delta A, \Delta {B}} {\|\Delta A\|_{\rm F}}^{2}+{\|\Delta {B}\|_{\rm F}}^{2}
\quad {\rm such \ that}\  \begin{bmatrix}
{A}_{1} \\
 {A}_{2}+\Delta A_{2} 
\end{bmatrix}{X}_{\rm ctls}=
\begin{bmatrix}
B_{1}
\\
B_2+\Delta {B}_{2}
\end{bmatrix},
\end{align}
where the first $k$ rows of $A$ and $B$ are fixed.
This constrained TLS problem~\eqref{eq:ctls2} can be solved by the orthogonal
transformations; see~\cite{LJY2022} for the details.

Here we consider the corresponding regression model
described as follows:
\begin{align}\label{eq:newmodel}
\begin{bmatrix}
{A}_1 \\
\bar{A}_2 
\end{bmatrix}
{X}=\begin{bmatrix}
{B}_1 \\
\bar{B}_2 
\end{bmatrix}
,\quad 
A=\begin{bmatrix}
{A}_1 \\ 
\bar{A}_2+E_{A_2} 
\end{bmatrix}
,\quad 
{B}=\begin{bmatrix}
{B}_1 \\
\bar{B}_2+E_{B_{2}} 
\end{bmatrix},
\end{align}
where all the elements of $E_{A_2}$ and ${E}_{{B}_2}$
are random variables, and only $A$ and ${B}$ can be observed.
Related to this formulation, 
we assume ${\rm rank}(A_{1})=k<n$.
If not, we select independent rows in ${A}_{1}$ to recover $X$.
Thus, ${\rm rank}(A_{1})=k<n$ is not restrictive.
For the random matrices, let
\begin{align}\label{eq:E3}
E:=
\begin{bmatrix}
E_{A_2} & {E}_{{B}_2} 
\end{bmatrix}.
\end{align}
In addition, similarly to~\eqref{eq:barC}, let
\begin{align}\label{eq:barC2}
\bar{C}_{2}:=
\begin{bmatrix}
\bar{A}_2 & \bar{B}_2 
\end{bmatrix}.
\end{align}
Moreover, with the range null space decomposition in mind,
define $P\in \R^{(n+\ell)\times (n+\ell-k)}$ such that
\begin{align}\label{eq:P}
P^{\top}P=I_{n+\ell -k},
\quad
\begin{bmatrix}
A_1 & {B}_1 
\end{bmatrix}P=0_{k\times (n+\ell-k)}.
\end{align}
The matrix $P$ corresponds to the orthonormal basis vectors
of the orthogonal complementary subspace of $[A_1 \ \ {B}_1]$.
We extend Assumption~\ref{as:CTC} to the following form.
\begin{assumption}\label{as:PCCP}
$\lim_{m\to \infty}m^{-1}{P}^{\top}{\bar{C}_{2}}^{\top}\bar{C}_{2}{P}=:\bar{T}$ exists, and ${\rm rank}(\bar{T})=n-k$. 
\end{assumption}

Since the case of $k=0$ corresponds to $P=I$,
Assumption~\ref{as:PCCP} for $k=0$ is reduced to 
Assumption~\ref{as:CTC}.
The next theorem states the strong consistency of the estimator ${X}_{\rm ctls}$ in \eqref{eq:ctls2}.
\begin{theorem}[Strong consistency~{\cite[Theorem~2]{Aishima2022}}]\label{thm:Aishima}
For $\ell=1$ and a fixed $k$, 
under Assumptions~\ref{as:E} and~\ref{as:PCCP},
we have
\begin{align*}
\lim_{m\to \infty}{X}_{\rm ctls}={X} \ \text{with probability one},
\end{align*}
where ${X}_{\rm ctls}$ is the solution vector of \eqref{eq:ctls2}.
\end{theorem}

For the general $\ell \ge 1$, we prove the strong consistency below.
In such a multidimensional case,
the condition number and the solvability for the row constraint~\eqref{eq:ctls2} 
are given in detail in~\cite{LJ2022,LJY2022}.
The contribution of this study 
is to prove strong consistency
of $X_{\rm ctls}$ in the statistical sense
in the more general cases
under reasonable assumptions.

Regarding the consistency analysis,
the following open questions remained unresolved
in the previous study.

\begin{itemize}
\item
In the rank deficient case $r < n$ in~\eqref{eq:rankdeficientA},
the consistency of a set of the solution $X$ is unknown,
though the minimal norm solution is the consistent estimator
as in Theorem~\ref{thm:park} \cite[Lemma~4.6]{Park2011}

\item
For the regression model with error free parts $A_{1}$ and $B_{1}$ as in \eqref{eq:newmodel},
the rank deficient case of the coefficient matrix  $\left[ A_{1}{}^{\top} \ \ \bar{A}_{2}{}^{\top} \right]^{\top}$ is not discussed

\end{itemize}

This paper is a unified answer to the above two questions. We prove the strong consistency of an estimator for obtaining a set of $X$, and hence the strong consistency of the minimal norm solution is clear.

\section{Formulation of a statistical model related to the optimization problem}\label{sec:aim}
With the research background in the previous section, 
we formulate the target regression model
to construct an estimator with consistency analysis. 
Basically, we generalize the discussions in Section~\ref{ssec:ctls},
focusing in particular on the conditions regarding the rank of the coefficient matrix.
Hence, for the linear regression model~\eqref{eq:newmodel}, we assume
\begin{align}\label{eq:A11A12}
\bar{A}:=
\begin{bmatrix}
{A}_{1} \\ 
\bar{A}_{2}
\end{bmatrix},
\quad
{\rm rank}
\left(
\bar{A}
\right)
=r \le n.
\end{align}
Then the solution $X$ in~\eqref{eq:newmodel} is not unique.
Below is a detailed analysis of the properties of $X$.

Without loss of generality, for ${A}_{1}\in \R^{k\times n}$,
we assume
\begin{align}\nonumber
{\rm rank}\left(
{A}_{1} 
\right)
=k<r.
\end{align}
In other words, ${A}_{1}$ is a row full rank rectangular matrix.
If not, we select independent rows in $A_{1}$ to recover $X$.
Thus, the rank requirement above is not restrictive.

In the rank deficient case $r<n$ in~\eqref{eq:A11A12},
 $X$ in~\eqref{eq:newmodel} is not unique and the solution set constitutes 
a certain subspace.
Similarly to ${Y}$ in~\eqref{eq:CY}, for $r<n$, 
with the range null space decomposition of $\bar{A}$ in mind,
we introduce $\bar{Y}$ to
express $X$ as in the next lemma.

\begin{lemma}\label{lem:barYX}
Define $W\in\R^{n \times (n-r)}$ and $\bar{Y}\in \R^{(n+\ell)\times (n+\ell -r)}$
such that
\begin{align}\label{eq:subspace}
\Ra(W)=\Nu(\bar{A}) \Longleftrightarrow
\bar{A}W=
\begin{bmatrix}
{A}_{1} \\ 
\bar{A}_{2}
\end{bmatrix}
W=0_{m\times (n-r)},
\quad
\bar{Y}:=
\begin{bmatrix}
-X_{\min} & W \\
 I_{\ell} & 0_{\ell \times (n-r)}
\end{bmatrix},\quad
\Ra(\bar{Y})=\Nu \left(
\begin{bmatrix}
{A}_{1} & {B}_{1} \\ 
\bar{A}_{2} & \bar{B}_{2}
\end{bmatrix}
\right),
\end{align}
where $X_{\min}$ is the minimal norm solution.
Then any solution $X$ can be given by
\begin{align}\label{eq:anyX}
L\in \R^{(n-r)\times \ell},\quad 
X=X_{\min}+WL.
\end{align}
\end{lemma}

In the above lemma, ${\rm rank}(W)=n-r$.
In the case of $r=n$,
${X}$ is unique, and thus
$\bar{Y}$ is uniquely determined without $W$.
Note that since the coefficient matrix $\bar{A}$ is unknown,
$W$ is also unknown, which makes the problem not easy.
For the convergence analysis, the range $\Ra(\bar{Y})$ plays an important role,
leading to any $X$. 

Similar to~Assumption~\ref{as:PCCP}, we assume the following,
including the condition that the rank is $r$ .

\begin{assumption}\label{as:PCCP2}
$\lim_{m\to \infty}m^{-1}{P}^{\top}{\bar{C}_{2}}^{\top}\bar{C}_{2}{P}=:\bar{T}$ exists, and ${\rm rank}(\bar{T})=r-k$. 
\end{assumption}

For estimating $X$, we consider the following optimization problem:
\begin{align}\label{eq:ctls2rank}
\min_{{X}_{\rm ctls}, \Delta A_{2}, \Delta {B}_{2}} {\|\Delta A_{2}\|_{\rm F}}^{2}+{\|\Delta {B}_{2}\|_{\rm F}}^{2}
\quad {\rm such \ that}\  \begin{bmatrix}
{A}_{1} \\
 {A}_{2}+\Delta A_{2} 
\end{bmatrix}{X}_{\rm ctls}=
\begin{bmatrix}
B_{1}
\\
B_2+\Delta {B}_{2}
\end{bmatrix},\quad 
{\rm rank}\left(
\begin{bmatrix}
{A}_{1} \\ 
{A}_{2}+\Delta A_{2} 
\end{bmatrix}
\right)=r.
\end{align}

Below we consider the above estimator $X_{\rm ctls}$
for the linear regression model as in~\eqref{eq:newmodel}.

\begin{remark}
For the standard linear regression model,
without errors-in-variables,
there is a rank deficient case for $X$
as in~\cite{MCWZ2015}. 
If we formulate such a situation, 
the optimization problem is
\begin{align}\nonumber
\min_{{X}_{\rm ctls}, \Delta A_{2}, \Delta {B}_{2}} {\|\Delta A_{2}\|_{\rm F}}^{2}+{\|\Delta {B}_{2}\|_{\rm F}}^{2}
\quad {\rm such \ that}\  \begin{bmatrix}
{A}_{1} \\
 {A}_{2}+\Delta A_{2} 
\end{bmatrix}{X}_{\rm ctls}=
\begin{bmatrix}
B_{1}
\\
B_2+\Delta {B}_{2}
\end{bmatrix},\quad 
{\rm rank}\left( X_{\rm ctls} \right)=r<\ell.
\end{align}
The corresponding linear regression model
is to include the rank constraint 
${\rm rank}\left( X \right)=r$
into \eqref{eq:newmodel}.

In fact, the estimator $X_{\rm ctls}$ in \eqref{eq:ctls2rank}
for the regression model in \eqref{eq:newmodel}
has strong consistency
even if the above rank constraint for $X$ exists.
Noting Lemma~\ref{lem:barYX}, we see that
$\bar{Y}$ in~\eqref{eq:subspace}
is a full column rank matrix,
even if ${\rm rank}(X_{\min})<\ell$.
Such a rank constraint does not matter for the proof of strong consistency.
Below we prove that the consistency of $X_{\rm ctls}$
does not depend on any constraint of $X$.
Since this study focuses on the consistency
without the discussion of the accuracy of the estimator,
the detailed analysis for the above reduced rank estimator 
is beyond the scope of this paper.
\end{remark}

For the statistical asymptotic analysis, 
we define some general statistical terms and concepts exactly.
In general, the strong convergence of random variables
is defined as follows.

\begin{definition}[Strong convergence]\label{def:sc}
The sequence of random variables $\mathcal{X}^{(m)}$
for $m=0,1,\ldots$ converges to $\mathcal{X}$ with probability one
if and only if
\begin{align*}
\mathcal{P}(\{ \omega \in \Omega \mid \lim_{m\to \infty}\mathcal{X}^{(m)}(\omega)=\mathcal{X}(\omega) \})=1,
\end{align*}
where $\mathcal{P}$ is a probability measure on a sample space $\Omega$.
\end{definition}

In this paper, the random variable argument $\omega$ is omitted for simplicity.
Let $\widehat{\mathcal{X}}^{(m)}$ denote 
an estimator of a target constant value $\theta$
with the use of a sample of size $m$.
In other words, 
$\mathcal{X}(\omega)$ in Definition~\ref{def:sc}
at the limit point is a constant $\theta$ independent of $\omega$.
If $\widehat{\mathcal{X}}^{(m)}$
converges to $\theta$ with probability one,
the estimator $\widehat{\mathcal{X}}^{(m)}$
is said to be strongly consistent.
Then $\widehat{\mathcal{X}}^{(m)}$ has the strong consistency
as in the following definition.

\begin{definition}[Strong consistency]\label{def:sconsist}
Suppose that the sequence of random variables $\widehat{\mathcal{X}}^{(m)}$
for $m=0,1,\ldots$ converges to a constant $\theta$ with probability one.
Then $\widehat{\mathcal{X}}^{(m)}$ is called an estimator of $\theta$
with strong consistency.
\end{definition}

We aim to prove that $X_{\rm ctls}$ has the strong consistency as $m\to \infty$. 
For the consistency analysis of $X_{\rm ctls}$, we use the next lemma,
essentially proved in~\cite[Lemma~3.1]{Gleser1981}.
The lemma is also due to the Kolmogolov strong law of large numbers.

\begin{lemma}\label{lem:kolmogolov}
For a sequence of real numbers $\{\alpha_{i}\}_{i=1}^{\infty}$
and any integer $m\ge 1$, suppose that
$m^{-1}\sum_{i=1}^{m}{\alpha_{i}}^2$ are bounded, i.e.,
$m^{-1}\sum_{i=1}^{m}{\alpha_{i}}^2 < \infty$ for all $m$.
In addition,
let $\{\epsilon^{(i)}\}_{i=1}^{\infty}$ denote a sequence of 
independent and identically distributed (i.i.d.) random variables
with $\mathbb{E}(\epsilon^{(i)})=0$ for every $i$,
where $\mathbb{E}$ means the expectation of random variables.
Moreover, suppose that, for all $i$, 
the common variance $\sigma^2:=\mathbb{E}({\epsilon^{(i)}}^{2})$
is bounded.
Then 
\begin{align}\label{eq:keyslln}
\lim_{m\to \infty}m^{-1}\sum_{i=1}^{m}\alpha_{i}\epsilon^{(i)}=0 \ \text{with probability one}.
\end{align}
\end{lemma}
\begin{proof}
Noting Abel's partial summation formula, we define $\beta_{i} \ (i=0,1,\ldots)$ such that
\begin{align*}
\beta_{0}=0,\quad \beta_{m}=\sum_{i=1}^{m}{\alpha_{i}}^{2}\quad (m=1,2,\ldots),
\end{align*}
leading to
\begin{align}\label{eq:beta}
m^{-1}\beta_{m}<\infty \quad (m=1,2,\ldots)
\end{align}
from the assumption that $m^{-1}\sum_{i=1}^{m}{\alpha_{i}}^2$ are bounded.
Then we have
\begin{align*}
\sum_{i=1}^{m}i^{-2}{\alpha_{i}}^{2}=
\sum_{i=1}^{m}i^{-2}(\beta_{i}-\beta_{i-1})=
\sum_{i=1}^{m}i^{-2}\beta_{i}-\sum_{i=0}^{m-1}(i+1)^{-2}\beta_{i}=
m^{-2}\beta_{m}-\sum_{i=1}^{m-1}(i^{-2}-(i+1)^{-2})\beta_{i}.
\end{align*}
The second term on the right-hand side of the above equation is
$\sum_{i=1}^{m-1}(i^{-1}(i+1)^{-2}(2i+1))i^{-1}\beta_{i}$.
In addition,
\eqref{eq:beta} and 
$\sum_{i=1}^{m-1}i^{-1}(i+1)^{-2}(2i+1)<\infty$
result in $\sum_{i=1}^{m}i^{-2}{\alpha_{i}}^{2}<\infty$.
Thus, from the assumption that $\mathbb{E}(\epsilon^{(i)}{}^{2})=\sigma^2<\infty \ (i=1,2,\ldots )$, 
we obtain
\begin{align*}
\sum_{i=1}^{m}i^{-2}\mathbb{E}((\alpha_{i}\epsilon^{(i)})^{2})=
\sigma^2\sum_{i=1}^{m}i^{-2}{\alpha_{i}}^{2}<\infty,
\end{align*}
yielding \eqref{eq:keyslln}
by \cite[eq. (7) in Corollary to Theorem~5.4.1]{Chung2001}.
\end{proof}

Under the above mathematical preparation,
we prove strong consistency in the next section.

\section{Consistency analysis for the CTLS solution}\label{sec:mainresult}
In this section, we prove strong consistency,
including the rank deficient case, which causes the lack of the uniqueness of $X$.
Since the estimator $X_{\rm ctls}$ is the solution of \eqref{eq:ctls2rank},
$X_{\rm ctls}$ is not unique, and hence, the meaning of the consistency is unclear.
First of all, to clarify its meaning,
we review the optimization algorithm,
focusing on the feature of projection,
as in~\cite{Demmel1987,LJY2022}.

\subsection{Projection algorithm for the optimization}
Similarly to~\eqref{eq:barC2}, let
\begin{align}\label{eq:C2}
{C}_{2}:=
\begin{bmatrix}
{A}_2 & {B}_2 
\end{bmatrix}.
\end{align}
Using the above $C_{2}$ and $P$ in \eqref{eq:P}, we compute $G$ such that
\begin{align}\label{eq:defG}
G=P^{\top}{{C}_{2}}^{\top}C_{2}P \in \R^{(n+\ell -k )\times (n+\ell -k )}.
\end{align}
In addition, let $V  \in \R^{(n+\ell -k )\times (n+\ell -r)}$ 
denote an eigenvector matrix comprising
the orthonormal eigenvectors corresponding to the smallest $n+\ell -r$
eigenvalues of $G$. Defined on the basis of the SVD, 
the column vectors of $V$ are the right singular vectors of $C_{2}P$
corresponding to its smallest singular values.

Using $V$, we define
\begin{align}\label{eq:Z}
Z=PV,\quad
Z=:
\begin{bmatrix}
Z_{{\rm upper}}
\\
Z_{{\rm lower}}
\end{bmatrix},\ 
Z_{{\rm upper}}\in \R^{n\times (n+\ell -r)},\
Z_{{\rm lower}}\in \R^{\ell \times (n+\ell -r)}.
\end{align}
Then, for $Z_{{\rm lower}}\in \R^{\ell \times (n+\ell -r)}$,
the Moore-Penrose generalized matrix inverse $Z_{{\rm lower}}{}^{\dagger}$
is often applied due to $\ell \le n+\ell -r$.
As proved in the following sections, 
$Z_{{\rm lower}}$ is a row full rank matrix for sufficiently large $m$
with probability one.
Thus, ${Z_{{\rm lower}}}^{\dagger}$ can be computed by
\begin{align}\nonumber
{Z_{{\rm lower}}}^{\dagger}={Z_{{\rm lower}}}^{\top}({Z_{{\rm lower}}}{Z_{{\rm lower}}}^{\top})^{-1}.
\end{align}
Then we obtain
\begin{align}\label{eq:Xctls*}
Z{Z_{{\rm lower}}}^{\dagger}=
\begin{bmatrix}
Z_{{\rm upper}}{Z_{{\rm lower}}}^{\dagger}
\\
I_{\ell}
\end{bmatrix},\quad
X_{\rm ctls}^{*}=-Z_{{\rm upper}}{Z_{{\rm lower}}}^{\dagger},
\end{align}
where $X_{\rm ctls}^{*}$ is the minimal norm solution.

The above computational procedure is summarized as follows.

\begin{itemize}
\item
Form $G$ in \eqref{eq:defG} 

\item
Compute the eigenvector matrix $V$ 
corresponding to the $n+\ell -r$ smallest 
eigenvalues of $G$

\item
Compute $Z$ in \eqref{eq:Z}, yielding $X_{\rm ctls}^{*}$ as in \eqref{eq:Xctls*}
\end{itemize}

The computation of $Z=PV$ in \eqref{eq:Z} corresponds to 
the Rayleigh-Ritz procedure~\cite{Saad2011} for $C_{2}{}^{\top}C_{2}$ using $\Ra(P)$,
where $C_{2}$ and $P$ are defined as in \eqref{eq:C2} and \eqref{eq:P}, respectively.
In other words, the column vectors of $Z$
are the Ritz vectors.
Our goal is to derive convergence as a set of solutions, 
not just a minimal norm solution, 
and to that end we analyze $\Ra(Z)$,
which converges in some sense to $\Ra(\bar{Y})$ in \eqref{eq:subspace}, in the following discussions.

\subsection{Deterministic asymptotic analysis without random noise}
The computed matrix $Z$ in \eqref{eq:Z} plays an important role
in the convergence analysis.
Here we consider $\bar{Z}\in \R^{(n+\ell)\times (n+\ell -r)}$
corresponding to the noiseless matrix data.

For the definition of $\bar{Z}$, let
\begin{align}\label{eq:defGbar}
\bar{G}=P^{\top}{\bar{C}_{2}}{}^{\top}\bar{C}_{2}P \in \R^{(n+\ell -k )\times (n+\ell -k )}
\end{align}
in the same manner as $G$ in \eqref{eq:defG},
where $\bar{C}_{2}$ and $P$ are defined as in \eqref{eq:barC2} and \eqref{eq:P}.
Under Assumption~\ref{as:PCCP2}, we have
${\rm rank}(\bar{G})=r-k$ for all sufficiently large $m$.
Hence, the dimension of $\Nu(\bar{G})$ is $n+\ell -r$,
implying that $\bar{G}$ has the $n+\ell -r$ eigenvectors
corresponding to the zero eigenvalues.
Let $\bar{V} \in \R^{(n+\ell -k )\times (n+\ell -r)}$ 
denote an eigenvector matrix comprising
the orthonormal eigenvectors corresponding to $\Nu(\bar{G})$, i.e., the null space of $\bar{G}$.
Similarly to \eqref{eq:Z}, let
\begin{align}
\bar{Z}=P\bar{V},\quad
\bar{Z}=:
\begin{bmatrix}
\bar{Z}_{{\rm upper}}
\\
\bar{Z}_{{\rm lower}}
\end{bmatrix},\ 
\bar{Z}_{{\rm upper}}\in \R^{n\times (n+\ell -r)},\
\bar{Z}_{{\rm lower}}\in \R^{\ell \times (n+\ell -r)}.
\end{align}
Below we consider $\bar{Z}$,
related to $\bar{Y}$ in \eqref{eq:subspace}.

First of all, we have
\begin{align}
\begin{bmatrix}
{A}_{1} & {B}_{1}
\\
\bar{A}_{2} & \bar{B}_{2}
\end{bmatrix}\bar{Z}
=
\begin{bmatrix}
{A}_{1} & {B}_{1}
\\
\bar{A}_{2} & \bar{B}_{2}
\end{bmatrix}P\bar{V}
=0_{m\times (n+\ell -r)},
\quad
{\rm rank}(\bar{Z})=n+\ell -r
\end{align}
due to the features of $P$ in \eqref{eq:P} and 
the eigenvector matrix $\bar{V}$ corresponding to $\Nu(\bar{G})$.
Combined this with \eqref{eq:subspace},
we have
\begin{align}\label{eq:RaZYNuA}
\Ra(\bar{Z})=\Ra(\bar{Y})=
\Nu \left(
\begin{bmatrix}
{A}_{1} & {B}_{1}
\\
\bar{A}_{2} & \bar{B}_{2}
\end{bmatrix}\right).
\end{align}

Noting $\Ra(\bar{Z})=\Ra(\bar{Y})$,
we have some nonsingular matrix $M$ such that
\begin{align}\label{eq:ZMY}
\begin{bmatrix}
\bar{Z}_{{\rm upper}}
\\
\bar{Z}_{{\rm lower}}
\end{bmatrix}M
=
\begin{bmatrix}
-X_{\min} & W \\
 I_{\ell} & 0_{\ell \times (n-r)}
\end{bmatrix}.
\end{align}
Thus, $\bar{Z}_{{\rm lower}}$
is the row full rank, i.e., ${\rm rank}(\bar{Z}_{{\rm lower}})$.
Therefore, there exists $({\bar{Z}_{{\rm lower}}}{\bar{Z}_{{\rm lower}}}{}^{\top})^{-1}$,
and thus the Moore-Penrose generalized matrix inverse is 
\begin{align}\nonumber
{\bar{Z}_{{\rm lower}}}{}^{\dagger}={\bar{Z}_{{\rm lower}}}{}^{\top}({\bar{Z}_{{\rm lower}}}{\bar{Z}_{{\rm lower}}}{}^{\top})^{-1}.
\end{align}
This leads to
\begin{align}\nonumber
X=-\bar{Z}_{{\rm upper}}{\bar{Z}_{{\rm lower}}}{}^{\dagger}.
\end{align}
Below we prove that $-\bar{Z}_{{\rm upper}}{\bar{Z}_{{\rm lower}}}{}^{\dagger}$
is the minimal norm solution.

Noting $\bar{Z}^{\top}\bar{Z}=I_{n+\ell -r}$, we have
\begin{align}\nonumber
\|-\bar{Z}_{{\rm upper}}{\bar{Z}_{{\rm lower}}}{}^{\dagger}\|_{\rm F}{}^{2}
=\|\bar{Z}{\bar{Z}_{{\rm lower}}}{}^{\dagger}\|_{\rm F}{}^{2}-\ell 
=\|{\bar{Z}_{{\rm lower}}}{}^{\dagger}\|_{\rm F}{}^{2}-\ell.
\end{align}
The generalized matrix inverse is not unique.
However, since the Moore-Penrose generalized matrix inverse is the minimal norm one,
$-\bar{Z}_{{\rm upper}}{\bar{Z}_{{\rm lower}}}{}^{\dagger}$
is the minimal norm solution, i.e., 
\begin{align}\label{eq:ZXmin}
-\bar{Z}_{{\rm upper}}{\bar{Z}_{{\rm lower}}}{}^{\dagger}=X_{\min}.
\end{align}
To clarify the relationship between $\bar{Y}$ and $\bar{Z}$,
define $\bar{Z}_{\rm lower}^{\perp}\in \R^{(n+\ell-r) \times (n-r)}$ such that
\begin{align}\nonumber
{\rm rank}(\bar{Z}_{\rm lower}^{\perp})=n-r,\quad
\bar{Z}_{\rm lower}\bar{Z}_{\rm lower}^{\perp}=0_{\ell \times (n-r)} \Longleftrightarrow \Nu(\bar{Z}_{\rm lower})=\Ra(\bar{Z}_{\rm lower}^{\perp}).
\end{align}
Thus, substituting $M:=[\bar{Z}_{\rm lower}{}^{\dagger} \ \ \bar{Z}_{\rm lower}^{\perp}]$ in \eqref{eq:ZMY},
we have
\begin{align}\nonumber
\bar{Z}_{\rm upper}\bar{Z}_{\rm lower}^{\perp}=W,
\end{align}
where $W$ is a submatrix of $\bar{Y}$ in \eqref{eq:subspace} as in Lemma~\ref{lem:barYX}.
Thus, any $X$ can be obtained by~\eqref{eq:anyX}.
In summary, it is necessary and sufficient to find $\bar{Z}$.

From the above discussions, 
we aim to estimate $\bar{Z}$
to identify the set of solutions,
containing the minimal norm solution as in~\eqref{eq:ZXmin}.
More precisely, we prove the convergence of $V$ to $\bar{V}$,
resulting in the convergence of $Z(=PV)$ to $\bar{Z}(=P\bar{V})$ in the next section.

\subsection{Statistical asymptotic analysis for noisy data}
In this section, we prove strong consistency of the estimator $X_{\rm ctls}$.
To this end, recall that
$E$ is defined as in~\eqref{eq:E},
also for the new regression model in~\eqref{eq:newmodel}.
The next lemma is 
interpreted as a generalization of~\cite[Lemma~3.1]{Gleser1981} with its proof
to projected subspace by $P$ in \eqref{eq:P}.

\begin{lemma}\label{lem:key}
Under Assumptions~\ref{as:E} and~\ref{as:PCCP2},
we have 
\begin{align}\label{eq:key}
\lim_{m\to \infty}m^{-1}G=
\bar{T}+\sigma^2 I_{n+\ell-k}
\end{align}
with probability one.
\end{lemma}
\begin{proof}
Firstly, we focus on
\begin{align}\label{eq:exrandGP}
m^{-1}G=m^{-1}P^{\top}{C_{2}}^{\top}C_{2}P=m^{-1}P^{\top}(\bar{C}_{2}+E)^{\top}(\bar{C}_{2}+E)P,
\end{align}
where the first equality is due to \eqref{eq:defG},
and the second equality is due to~\eqref{eq:E3}, \eqref{eq:barC2}, and \eqref{eq:C2}.
The right-hand side is 
\begin{align*}
m^{-1}P^{\top}{\bar{C}}_{2}^{\top}\bar{C}_{2}P+m^{-1}P^{\top}{\bar{C}_{2}}^{\top}EP 
+m^{-1}P^{\top}E^{\top}\bar{C}_{2}P+m^{-1}P^{\top}E^{\top}EP.
\end{align*}
Thus, we evaluate each term in turn.

Under Assumption~\ref{as:PCCP2},
\begin{align}\label{eq:PGGP}
\lim_{m\to \infty}m^{-1}P^{\top}{\bar{C}_{2}}^{\top}\bar{C}_{2}P=\bar{T}
\end{align}
is obvious.
In addition, the strong law of large numbers under Assumption~\ref{as:E} yields
\begin{align*}
\lim_{m\to \infty}m^{-1}E^{\top}E=\sigma^2 I_{n+\ell }
\end{align*}
with probability one.
From the orthogonality $P^{\top}P=I_{n+\ell -j}$ as in \eqref{eq:P}, we have
\begin{align}\label{eq:PEEP}
\lim_{m\to \infty}m^{-1}P^{\top}E^{\top}EP=\sigma^2 I_{n+\ell-k}
\end{align}
with probability one.

In the following, we prove
\begin{align}\label{eq:PEGP}
\lim_{m\to \infty}m^{-1}P^{\top}E^{\top}\bar{C}_{2}P=0_{(n+\ell -k)\times (n+\ell -k)} 
\end{align}
with probability one,
implying the convergence of $m^{-1}P^{\top}{\bar{C}_{2}}^{\top}EP$.
To this end, letting $\bar{U}=\bar{C}_{2}P \in \R^{m \times (n+\ell -k)}$, we have
\begin{align*}
\lim_{m\to \infty}m^{-1}\bar{U}^{\top}\bar{U}=\bar{T} \in \R^{(n+\ell -k) \times (n+\ell -k)} 
\end{align*}
under Assumption~\ref{as:PCCP2}.
Let $\bar{u}_{ij}\ (1\le i \le m,\ 1\le j \le n+\ell -k)$ denote the $(i,j)$ elements of $\bar{U}$.
Then $\lim_{m\to \infty}m^{-1}\sum_{i=1}^{m}{\bar{u}_{ij}}^{2}$ exist for all $j$.
Similarly, let $e_{ij'}\ (1\le i \le m,\ 1\le j' \le n+\ell )$ denote the $(i,j')$ elements of $E$.
Here we use Lemma~\ref{lem:kolmogolov}.
For fixed $j$ and $j'$,
the above random variables $e_{ij'}\bar{u}_{ij}\ (i=1,2,\ldots ,m)$ are all independent
because $e_{ij'}$ are i.i.d. for all $i$.
Note that $\bar{u}_{ij}\ (i=1,2,\ldots ,m)$ correspond to
$\alpha_{i}\ (i=1,\ldots ,m)$ in Lemma~\ref{lem:kolmogolov}.
Therefore, for $1\le j \le n+\ell -k$ and $1 \le j' \le n+\ell$,
\begin{align*}
\lim_{m\to \infty}m^{-1}\sum_{h=1}^{m} e_{ij'}\bar{u}_{ij}=0
\end{align*}
with probability one.
The matrix expression is
\begin{align*}
\lim_{m\to \infty}m^{-1}E^{\top}\bar{U}=\lim_{m\to \infty}m^{-1}E^{\top}\bar{C}_{2}P=0_{(n+\ell -k)\times (n+\ell -k)}
\end{align*}
with probability one,
leading to~\eqref{eq:PEGP}.

Therefore, using~\eqref{eq:exrandGP}, \eqref{eq:PGGP}, \eqref{eq:PEEP}, and \eqref{eq:PEGP},
we obtain \eqref{eq:key}
with probability one.
\end{proof}

Importantly, the eigenvector matrix $V$ of $G$
converges to $\bar{V}$ that is the eigenvector matrix of $\bar{T}$ from Lemma~\ref{lem:key},
even though $m^{-1}G$ does not converge to $\bar{T}$ as in Lemma~\ref{lem:key}.
However, since the eigenvector matrix $\bar{V}$ is not unique,
we clarify the meaning of the convergence $\bar{V}$ rigorously below.

Although the eigenvector matrix $\bar{V}$ is not unique,
any eigenvector matrix can be represented by
$\bar{V}\bar{Q}$ for some orthogonal matrix $\bar{Q}\in \R^{(n+\ell -r)\times (n+\ell -r)}$.
Let $\bar{\mathcal{V}}$ denote such a set of the eigenvector matrices.
Then, from the continuities of eigenvectors and Lemma~\ref{lem:key}, we have
\begin{align}\nonumber
\lim_{m\to \infty}\min_{\bar{V}\in \bar{\mathcal{V}}} \|V-\bar{V}\|
\end{align}
with probability one.
Explicitly specifying the sample space $\Omega$
for the sake of mathematical rigor, the above convergence property
can be expressed as
\begin{align*}
\mathcal{P}(\{ \omega \in \Omega \mid \lim_{m\to \infty}\min_{\bar{V}\in \bar{\mathcal{V}}} \|V^{(m)}(\omega)-\bar{V}\| \})=1,
\end{align*}
where $\bar{\mathcal{V}}$ is independent of $\Omega$.

More precisely, $V$ is not unique either,
even if the eigenvalues are all distinct
in view of the arbitrariness of the sign.
However, we can define convergence as above 
by interpreting it as the representative element of the subspace
corresponding to the smallest eigenvalues,
resulting in the next crucial lemma.
We present here that the proof
based on the famous Davis-Kahan perturbation theorem in~\cite{DK1970},
while this property is clear from the continuities of the eigenvectors.

\begin{lemma}\label{lem:key2}
Under Assumptions~\ref{as:E} and~\ref{as:PCCP2},
we have 
\begin{align}\label{eq:key2}
\lim_{m\to \infty}\min_{\bar{V}\in \bar{\mathcal{V}}} \|V-\bar{V}\|
\end{align}
with probability one.
\end{lemma}
\begin{proof}
The Davis-Kahan perturbation theorem~\cite{DK1970} with Lemma~\ref{lem:key} implies, for any fixed positive $\epsilon < \sigma^2$, 
\begin{align}\nonumber
\|(I_{n+\ell -k}-\bar{V}\bar{V}{}^{\top})V\|\le \frac{(I_{n+\ell -k}-\bar{V}\bar{V}{}^{\top})(\bar{T}+\sigma^2I_{n+\ell -k}-m^{-1}G)V}{\sigma^2 -\epsilon}
\end{align}
with probability one for all sufficiently large $m$.
In addition, from the convergence property in Lemma~\ref{lem:key},
we have
\begin{align}\label{eq:VbarV}
\lim_{m\to \infty}\|(I_{n+\ell -k}-\bar{V}\bar{V}{}^{\top})V\|=
\lim_{m\to \infty}\|V-\bar{V}(\bar{V}{}^{\top}V)\|=0
\end{align}
with probability one.
Since $V$ and $\bar{V}$ comprise orthonormal column vectors, 
$V^{\top}V=\bar{V}^{\top}\bar{V}=I_{n+\ell -r}$ holds
for all $m$.
Thus, for $\bar{V}^{\top}V$ in~\eqref{eq:VbarV}, 
letting $\mathcal{Q}$ denote
a set of all $(n+\ell -r)\times (n+\ell -r)$ orthogonal matrices $Q$,
we have
\begin{align}\nonumber
\lim_{m\to \infty}\min_{Q\in {\mathcal{Q}}} \|\bar{V}^{\top}V-Q\|=0
\end{align}
with probability one.
Therefore, we obtain~\eqref{eq:key2},
where $\bar{\mathcal{V}}$ is a set of $\bar{V}$
comprising the orthonormal basis vectors of $\Nu(\bar{T})$.
\end{proof}

The above feature~\eqref{eq:key2} in Lemma~\ref{lem:key2} implies
\begin{align}
\Ra \left( \lim_{m\to \infty} V \right)=\Ra \left( \bar{V} \right) \ \text{with probability one}.
\end{align}
Recall that we aim to estimate $\bar{Y}$ in \eqref{eq:subspace} replaced with ${X}$
due to the lack of the uniqueness of ${X}$.
Note that $\Ra(\bar{Y})$ covers any $\bar{X}$.
The main theorem below states that
the estimator $Z:=PV$ identifies $\Ra(\bar{Y})$ with probability one
as $m\to \infty$.

\begin{theorem}\label{thm:new}
Under Assumptions~\ref{as:E} and~\ref{as:PCCP2},
we have
\begin{align}\label{eq:RaX}
\Ra \left( \lim_{m\to \infty}Z \right)=\Ra(\bar{Y}) \ \text{with probability one},
\end{align}
where $Z$ and $\bar{Y}$ are defined as in \eqref{eq:subspace} and \eqref{eq:Z}, respectively.
\end{theorem}
\begin{proof}
Recall that $\bar{V}$ comprises the orthonormal basis vectors
of $\Nu(\bar{T})$, resulting in $\Ra(P\bar{V})=\Ra(\bar{Z})=\Ra(\bar{Y})$.
From Lemma~\ref{lem:key2}, for $Z:=PV$, we obtain~\eqref{eq:RaX}.
\end{proof}

From the above theorem, the minimal norm solution
of \eqref{eq:ctls2rank} converges with probability one
as follows.

\begin{corollary}\label{cor:new}
Under Assumptions~\ref{as:E} and~\ref{as:PCCP2},
we have
\begin{align}\label{eq:xctlsx}
\lim_{m\to \infty}{X}_{\rm ctls}^{*}=X_{\min} \ \text{with probability one}.
\end{align}
\end{corollary}

For the usually performed consistency analysis of the minimal norm solution, we can obtain it as the above corollary of the main theorem as described above. The main novelty is that the main theorem, i.e., Theorem~\ref{thm:new}, also includes the consistency to $\Nu(\bar{A})$ for the exact coefficient matrix $\bar{A}$ in \eqref{eq:A11A12}.

From the point of view of finding $\Nu(\bar{A})$, 
it is also useful to compute the eigenvectors 
corresponding to the largest $r-k$ eigenvalues of $A^{\top}A$,
i.e., the truncated SVD of $A$.
While there may be other ways to do this, 
the contribution of this paper is to point out
that the computed $Z$ in the constrained TLS problem with the rank constraints~\eqref{eq:ctls2rank}
can recover all the solution vectors with probability one as $m \to \infty$.

\subsection{Rank deficient case in the asymptotic regime}
Here we discuss the rank deficient case in the asymptotic regime as $m\to \infty$.
In other words, we consider the following condition:
\begin{align}\label{eq:A11A12r}
{\rm rank}\left(
\begin{bmatrix}
{A}_{1} \\ 
\bar{A}_{2}
\end{bmatrix}
\right)
=r \le n,
\quad
\lim_{m\to \infty}m^{-1}\bar{G}=\bar{T},
\quad
{\rm rank}\left(\bar{T}\right)=r_{\infty} < r.
\end{align}

Define $\bar{Y}$ as in~\eqref{eq:subspace}
with the use of $r$.
Then any solution $X$ can be given by
\begin{align}\nonumber
L\in \R^{(n-r)\times \ell},\quad 
X=X_{\min}+WL.
\end{align}
Note that ${\rm rank}(m^{-1}\bar{G})=r$ for all sufficiently large $m$.
Let $\bar{V}\in \R^{(n+\ell -k)\times (n+\ell -r)}$ denote an eigenvector matrix
comprising the eigenvectors corresponding to the zero eigenvalues of $m^{-1}\bar{G}$.
Then we have
\begin{align}\nonumber
\bar{Z}:=P\bar{V},\quad \Ra(\bar{Z})=\Ra(\bar{Y})
\end{align}
in the same manner as~\eqref{eq:RaZYNuA}.
In the case of $r_{\infty}={\rm rank}(\lim_{m \to \infty}m^{-1}\bar{G})<r={\rm rank}(m^{-1}\bar{G})$
as in \eqref{eq:A11A12r},
the $r-r_{\infty}$ positive eigenvalues of $m^{-1}\bar{G}$ converge to 0 as $m\to \infty$.
Let $\bar{V}_{\infty}\in \R^{(n+\ell -k)\times (n+\ell -r_{\infty})}$ denote 
an eigenvector matrix corresponding to
the zero eigenvalues of $\bar{T}$.
Then, the continuities of the eigenvectors lead to
$\Ra(\bar{V})\subset \Ra(\bar{V}_{\infty})$, resulting in
\begin{align}\label{eq:Zinf}
\bar{Z}_{\infty}:=P\bar{V}_{\infty},\quad \Ra(\bar{Y}) \subset \Ra(\bar{Z}_{\infty}).
\end{align}
Thus, $\Ra(\bar{Z}_{\infty})$ can contain $\Ra(\bar{Y})$
for computing $X$.

On the basis of the above asymptotic analysis, 
we consider the following optimization problem:
\begin{align}\label{eq:ctls2r}
\min_{{X}_{\rm ctls}, \Delta A, \Delta {B}} {\|\Delta A\|_{\rm F}}^{2}+{\|\Delta {B}\|_{\rm F}}^{2}
\quad {\rm such \ that}\  \begin{bmatrix}
{A}_{1} \\
 {A}_{2}+\Delta A_{2} 
\end{bmatrix}{X}_{\rm ctls}=
\begin{bmatrix}
B_{1}
\\
B_2+\Delta {B}_{2}
\end{bmatrix},
\quad 
{\rm rank}\left(
\begin{bmatrix}
{A}_{1} \\ 
{A}_{2}+\Delta A_{2} 
\end{bmatrix}
\right)=r_{\infty}.
\end{align}

Although $r_{\infty}$ is not given for a finite $m$,
the rank should be estimated by the distribution of
the eigenvalues of $G$.
Regarding the eigenvalue distribution,
we have
\begin{align}\nonumber
\lim_{m\to \infty}m^{-1}G=\bar{T}+\sigma^{2}I_{n+\ell -k}
\end{align}
with probability one
in the same manner as in Lemma~\ref{lem:key}.
Thus, for sufficiently large $m$,
the $n+\ell -r_{\infty}$ smallest eigenvalues
appear to be nearly multiple,
interpreted as the bias due to the noise.
Thus, from the observed matrix $G$, the rank might be estimated by 
${\rm rank}(\bar{T})=r_{\infty}$.
In fact, the truncated TLS performs such a calculation.
In this case, we solve \eqref{eq:ctls2r}.
Thus, let $V  \in \R^{(n+\ell -k )\times (n+\ell -r_{\infty})}$ 
denote an eigenvector matrix comprising
the orthonormal eigenvectors corresponding to the smallest $n+\ell -r_{\infty}$
eigenvalues of $G$. 
Similarly to \eqref{eq:Z}, let
\begin{align}\label{eq:Zr}
Z:=PV \in \R^{(n+\ell )\times (n+\ell -r_{\infty})}.
\end{align}
The next theorem states that
$Z$ comprises basis vectors of an $n+\ell -r_{\infty}$ dimiensional subspace
including $\Ra(\bar{Y})$.

\begin{theorem}\label{thm:new3}
Under Assumptions~\ref{as:E} and~\ref{as:PCCP2},
we have
\begin{align}\label{eq:RaXsubset}
\Ra(\bar{Y}) \subset \Ra \left( \lim_{m\to \infty}Z \right) \ \text{with probability one},
\end{align}
where $\bar{Y}$ and $Z$ are defined as in \eqref{eq:subspace} and \eqref{eq:Zr}, respectively.
\end{theorem}
\begin{proof}
Due to the continuities of eigenvectors,
$\Ra(\lim_{m\to \infty}V)=\Ra(\bar{V}_{\infty})$ with probability one.
Thus, we have
\begin{align}\nonumber
\Ra(\lim_{m\to \infty}Z)=\Ra(\lim_{m\to \infty}PV)=\Ra(P\bar{V}_{\infty}) 
\end{align}
with probability one.
From~\eqref{eq:Zinf}, we obtain~\eqref{eq:RaXsubset}.
\end{proof}

From the above theorem, let
\begin{align}\nonumber
\begin{bmatrix}
Z_{{\rm upper}}
\\
Z_{{\rm lower}}
\end{bmatrix}:=Z,\ 
Z_{{\rm upper}}\in \R^{n\times (n+\ell -r_{\infty})},\
Z_{{\rm lower}}\in \R^{\ell \times (n+\ell -r_{\infty})}.
\end{align}
For all $X_{\rm ctls}$, using a generalized matrix inverse $Z_{\rm lower}{}^{\dagger}$, we have
\begin{align}\nonumber
X_{\rm ctls}=-Z_{{\rm upper}}Z_{\rm lower}{}^{\dagger},
\end{align}
where $Z_{\rm lower}{}^{\dagger}$ is not necessarily the Moor-Penrose matrix inverse.
Thus, $X_{\rm ctls}$ is not unique.
Let $\mathcal{X}_{\rm ctls}$ denote a set of $X_{\rm ctls}$.
Then we have
\begin{align}\nonumber
X\in \mathcal{X}_{\rm ctls}
\end{align}
for any $X$ of the linear regression model.
The lack of the required equations for solving $X$
in the asymptotic regime causes the above unsolvability.
However, we can identify a subspace including $X$
based on $Z \in \R^{(n+\ell) \times (n+\ell -r_{\infty})}$.

Such a feature indicates that,
if some error free equations are given,
they must be applied for enjoying the consistency, though a subspace including the solution can be identified. To clarify the feature, let us consider that the standard TLS in~\eqref{eq:tls} is applied to the regression model with equality constraints as in \eqref{eq:newmodel}.
Then, under Assumption~\ref{as:CTC2}, we see
\begin{align}\nonumber
\lim_{m\to \infty}m^{-1}\bar{A}^{\top}\bar{A}=
\lim_{m\to \infty}m^{-1}{A}_{1}{}^{\top}{A}_{1}+
\lim_{m\to \infty}m^{-1}\bar{A}_{2}{}^{\top}\bar{A}_{2}=
\lim_{m\to \infty}m^{-1}\bar{A}_{2}{}^{\top}\bar{A}_{2}
\end{align}
for a submatrix of $m^{-1}\bar{C}{}^{\top}\bar{C}$.
Since $k$ of $A_{1}\in \R^{k\times n}$ is fixed,
$\lim_{m\to \infty}m^{-1}{A}_{1}{}^{\top}{A}_{1}$ vanishes
as $m \to \infty$.
If the row vectors of $\bar{A}_{2}$ do not have the information
of $A_{1}$, $\lim_{m\to \infty}m^{-1}\bar{A}^{\top}\bar{A}$
is a rank deficient matrix. 
Thus, $X_{\rm tls}$ does not have the consistency.
To ensure the consistency, the computation of
$X_{\rm ctls}$ is required.
This feature also indicates the importance of 
the consistency analysis in this section.

\section{Conclusion}\label{sec:conclusion}
In this paper, the strong consistency
of the estimator by the TLS problem with rank constraints in~\eqref{eq:ctls2rank}
is proved.
Theorem~\ref{thm:new} states that
$\Ra(Z)$ converges to $\Ra(\bar{Y})$ with probability one
as $m\to \infty$, where $m$ is a sample size.
This indicates that a set of solutions $X$ 
can be identified by $\Ra(Z)$.
Thus, the consistency of the minimal norm solution
is obviously proved as in Corollary~\ref{cor:new}.
The result covers the standard TLS problem in~\eqref{eq:tls} 
as the special case $k=0$,
corresponding to $P=I_{n+\ell}$ in~\eqref{eq:P}. 
The statistical analysis here is based on the feature of orthogonal projection by $P$, or more specifically, on the characteristics of the Rayleigh-Ritz procedure, which is essentially useful for the analysis of TLS with row constraints. 
Thus, it leads to a straightforward proof of strong consistency for the solution set. The analysis for solution sets with other constraints and rank constraints is considered to be a future work.

\

\noindent
{\bf Acknowledgment}

\noindent
This study was supported by JSPS KAKENHI Grant Nos. JP21K11909 and JP22K03422.





\end{document}